\documentclass[oneside]{amsart}
\usepackage[utf8]{inputenc}
\usepackage{amsmath,amstext,amsopn,amssymb, amsfonts, amsthm, mathtools}
\usepackage[dvipsnames]{xcolor}
\usepackage{geometry}
\usepackage{bbm}
\usepackage{pdfsync}
\usepackage{tikz}
\usepackage{enumitem}

\usepackage[pdftex,bookmarks,colorlinks,breaklinks]{hyperref}  
\definecolor{dullmagenta}{rgb}{0.4,0,0.4}   
\definecolor{darkblue}{rgb}{0,0,0.4}
\definecolor{darkgreen}{rgb}{0,0.4,0}
\hypersetup{linkcolor=darkblue,citecolor=blue,filecolor=dullmagenta,urlcolor=darkblue} 

{}

\newtheorem*{definition*}{Definition}
\newtheorem{theorem}{Theorem}
\newtheorem*{theorem*}{Theorem}
\newtheorem{lemma}[theorem]{Lemma}
\newtheorem*{lemma*}{Lemma}

\newtheorem*{remark*}{Remark}
\numberwithin{equation}{section}
\numberwithin{theorem}{section}

\makeatletter
\newcommand{\customlabel}[2]{%
   \protected@write \@auxout {}{\string \newlabel {#1}{{#2}{\thepage}{#2}{#1}{}} }%
   \hypertarget{#1}{#2}
}
\makeatother

\def\XXint#1#2#3{{\setbox0=\hbox{$#1{#2#3}{\int}$}
     \vcenter{\hbox{$#2#3$}}\kern-.5\wd0}}

\newcommand{\sign}{\operatorname{sign}}

\makeindex
\begin{document}

\title{Another counterexample to Zygmund's conjecture}
\author{Guillermo Rey}
\address{}
\email{guillermo.reyley@gmail.com}

\maketitle
\begin{abstract}
  We present a simple dyadic construction that yields a new counterexample to Zygmund's conjecture.
  Our result recovers Soria's classical results in dimension three and four, through a different construction,
  and gives new ones in all other dimensions.
\end{abstract}

\setcounter{section}{-1}
\section{Introduction}

Lebesgue's Differentiation Theorem states that
\begin{equation*}
  \lim_{r \to 0} \frac{1}{2r}\int_{x-r}^{x+r} f(y)\, dy = f(x)
\end{equation*}
almost everywhere for all $f \in L^1(\mathbb{R})$.
The area of \emph{differentiation of integrals} explores these problems further,
the central question being whether a given basis
\footnote{For us a \emph{basis} will be a collection of bounded measurable sets.}
differentiates a given function space.
We say that a basis $\mathcal{B}$ differentiates a function space $X$
if for all $f \in X$
\begin{equation*}
  \lim_{k \to \infty} \frac{1}{|R_k|}\int_{R_k} f(y) \, dy = f(x) \quad a.e.,
\end{equation*}
where $\{R_k\}$ is any sequence of elements in $\mathcal{B}$ containing $x$ and with diameters converging to $0$.

Generalizing Lebesgue's Differentiation Theorem to higher dimensions can be done in various ways.
For example when we consider the basis consisting of all squares in $\mathbb{R}^2$ then the same result holds.
However, if we allow the squares to stretch into rectangles of arbitrary eccentricity then a naive generalization
fails spectacularly and one is forced to require $f$ to have better integrability properties.
We refer the reader to \cite{GuzmanBook} for a thorough overview of this area, and to \cite{Stokolos} for a review of some recent developments.

Perhaps the simplest basis for which non-trivial obstacles start to appear is that given by all rectangles with sides parallel
to the coordinate axes, or
in higher dimensions, the basis of \emph{intervals} in $\mathbb{R}^d$: cartesian products of one-dimensional intervals.
Here the Jessen-Marcinkiewicz-Zygmund theorem states that the basis of all $d$-dimensional intervals differentiates $L\log L^{d-1}$ (and this is sharp).

In \cite{Zygmund1967} A. Zygmund showed that one can obtain better estimates by restricting the ``degrees of freedom'' of a basis. In particular he showed
\begin{theorem*}[\cite{Zygmund1967}]
  The basis consisting of $d$-dimensional intervals whose sides have no more than $k$ different sizes differentiates $L\log L^{k-1}$.
\end{theorem*}

From this result a conjecture was born: given $d$ functions $\Phi_i : \mathbb{R}_+^k \to \mathbb{R}_+$ which are non-decreasing in each variable, the
basis consisting of all $d$-dimensional intervals with sidelengths
$$
  \Phi_1(t_1, \dots, t_k) \times \dots \times \Phi_d(t_1, \dots, t_k)
$$
differentiates $L\log L^{k-1}$.

In this setting, Zygmund's result shows that the conjecture is true when
$$
\Phi_i(t_1, \dots, t_k) = t_{\xi(i)} \quad \forall i
$$ where $\xi$ is any function on $\{1, \dots, d\}$ taking at most $k$ different values in $\{1, \dots, d\}$.

In \cite{CordobaZygmundsConjecture} A. C\'ordoba proved the conjecture for the special case of intervals in $\mathbb{R}^3$ of dimensions
$$
  s \times t \times \Phi(s,t),
$$
for any function $\Phi : \mathbb{R}^2_+ \to \mathbb{R}_+$ non-decreasing in each variable.

Then in \cite{SoriaCounterexample} F. Soria found a counterexample in dimension three. In particular, he constructs a basis of intervals of dimensions
$$
  s \times t \varphi(s) \times t \psi(s),
$$
with certain explicit increasing functions $\varphi$ and $\psi$ which cannot differentiate $L \log L$. He also constructs
a basis of intervals of the form
\begin{equation*}
  s \times t \times \Phi_1(s,t) \times \Phi_2(s,t)
\end{equation*}
which cannot differentiate $L \log L$ either.

In this article we construct various counterexamples extending the ones in \cite{SoriaCounterexample} to higher dimensions.
In the first section we construct bases of intervals with dimensions of the form
\begin{equation*}
  \Phi_1(s,t) \times \dots \times \Phi_d(s,t).
\end{equation*}
for all dimensions $d \geq 3$ and
which do not differentiate $L \log L^{d-2}$, thus showing that this version of Zygmund's conjecture fails even more dramatically in higher dimensions.

In the last section we show that
in dimensions four and higher one can also find bases with $\Phi_1(s,t) = s$ and $\Phi_2(s,t) = t$ which do not differentiate
$L \log L^{d-3}$.

\section{An extension lemma and the construction}

The counterexample will be purely dyadic, so it will be convenient to modify slightly the statement of Zygmund's conjecture.
Define a $k$-dimensional dyadic Zygmund
basis in $\mathbb{R}^d$ to be a collection $\mathcal{B}$ consisting of all dyadic intervals of dimensions
\begin{equation*}
  2^{\Phi_1(m)} \times \dots \times 2^{\Phi_d(m)} \quad \text{for some } m \in \mathbb{Z}^k,
\end{equation*}
where $\Phi_1, \dots, \Phi_d$ are functions from $\mathbb{Z}^d$ to $\mathbb{Z}$ which are non-decreasing in each variable.

In this setting, the dyadic Zygmund conjecture states that all $k$-dimensional dyadic Zygmund bases differentiate $L(\log^+ L)^{k-1}$.

A simple application of the one-third trick (see for example \cite{LNBook}) shows that this form of Zygmund's conjecture is equivalent to the one given in the introduction
as long as one allows ``dyadic interval'' to mean a possibly-shifted standard dyadic interval.

To construct the counterexample we need the following
\begin{lemma} \label{Lemma}
  Given any function $\varphi: \mathbb{Z} \to \mathbb{Z}$ there exists a function $\Phi : \mathbb{Z}^2 \to \mathbb{Z}$ satisfying
  \begin{enumerate}
    \item $\Phi(m,-m) = \varphi(m)$ for all $m \in \mathbb{Z}$.
    \item $\Phi$ is non-decreasing in each variable.
  \end{enumerate}
\end{lemma}
\begin{proof}
  For $m_1 \geq -m_2$ set
  \begin{equation*}
    \Phi(m_1, m_2) = \max\Bigl( \varphi(m):\, -m_2 \leq m \leq m_1\Bigr).
  \end{equation*}

  The set over which the maximum is taken increases as $m_1$ or $m_2$ increase, so $\Phi$ is non-decreasing in each variable on the right-hand side of the
  antidiagonal.
  On the left-hand side one can define $\Phi$ similarly: for $m_1 \leq -m_2$ set
  \begin{equation*}
    \Phi(m_1, m_2) = \min\Bigl( \varphi(m):\, m_1 \leq m \leq -m_2\Bigr),
  \end{equation*}
  and a similar argument applies.
\end{proof}

With this lemma we can now prove our main result
\begin{theorem} \label{Theorem}
  For any $2 \leq k \leq d$ there exist functions $\Phi_i : \mathbb{Z}^k \to \mathbb{Z}$ for $1 \leq i \leq d$
  which are non-decreasing in each
  variable, and for which the associated basis of dyadic rectangles does not differentiate
  $L(\log^+ L)^{d-2}$.
\end{theorem}
\begin{proof}
  Choose any bijection $\psi : \mathbb{Z} \to \mathbb{Z}^d$ and define for $1 \leq i \leq d$
  \begin{equation*}
    \varphi_i(m) = \psi_i(m) \quad \forall m \in \mathbb{Z}.
  \end{equation*}
  Now let $\Phi_i$ be the function provided by Lemma \ref{Lemma} applied to $\varphi_i$.

  By construction we have
  \begin{equation*}
    \{ (\Phi_1(m,-m), \dots, \Phi_d(m,-m) : m \in \mathbb{Z} \} = \mathbb{Z}^d,
  \end{equation*}
  so the collection of all dyadic intervals with dimensions
  \begin{equation*}
    2^{\Phi_1(m,n)} \times \dots \times 2^{\Phi_d(m,n)} \quad m,n \in \mathbb{Z}
  \end{equation*}
  is the collection of all dyadic intervals in $\mathbb{R}^d$.

  This proves the theorem when $k = 2$, but for larger $k$ one can just define
  $$
    \Psi_i(m_1, \dots, m_k) = \Phi_i(m_1, m_2)
  $$
  obtaining the same result.
\end{proof}

\begin{remark*}
  To put our result in context
  one should compare Theorem \ref{Theorem} with Soria's construction in \cite{SoriaCounterexample}.
  His construction yields counterexamples to Zygmund's conjecture in dimensions three and higher,
  so Theorem \ref{Theorem} gives an alternative construction in dimension three.

  In higher dimensions, the number of degrees of freedom of the functions $\Phi_i$ in Soria's construction depends linearly on the dimension,
  while in ours it depends only on two variables for all dimensions.
\end{remark*}

\section{Generalizations}

The previous section shows that any collection of dyadic intervals in $\mathbb{R}^d$ can be reproduced with intervals of dimensions
\begin{equation*}
  \Psi_1(s,t) \times \dots \times \Psi_d(s,t)
\end{equation*}
where all the functions $\Psi_i : \mathbb{R}_+^2 \to \mathbb{R}_+$ are non-decreasing in each variable.
In view of \cite{CordobaZygmundsConjecture} it is also of interest to obtain bases of intervals with dimensions of the form
\begin{equation*}
  t_1 \times \dots \times t_k \times \Psi_1(t_1, \dots, t_k) \times \dots \times \Psi_{d-k}(t_1, \dots, t_k).
\end{equation*}
One could be tempted to conjecture that such bases should differentiate $L \log L^{k-1}$. A. C\'ordoba's result shows
this is the case in dimension $d=3$ with $k=2$.

In Proposition 6 of \cite{SoriaCounterexample} A. Soria gave a counterxample to this conjecture with $d=4$ and $k=2$.
We now show how to adapt the ideas introduced in the previous section to obtain higher-dimensional generalizations of such counterexamples. In particular, we will
construct a basis of intervals with dimensions of the form
\begin{equation*}
  s \times t \times \Psi_1(s, t) \times \dots \times \Psi_{d-2}(s,t),
\end{equation*}
with $\Psi_i : \mathbb{R}_+^2 \to \mathbb{R}_+$ non-decreasing in both variables, which does not differentiate $L \log L^{d-3}$. Also like in
\cite{SoriaCounterexample}, the relevant elements of this basis are those with $s+t = 0$.

As before, we will work in the dyadic setting.
Fix a dimension $d \geq 4$ and for each $n \in \mathbb{N}$ let $i \mapsto \beta_{n,i}$ be the sequence of all
$(d-2)$-tuples $(m_1, \dots, m_{d-2}) \in \mathbb{Z}^{d-2}$ whose sum is $0$ and which satisfy $0 \leq m_j \leq n$ for all $1 \leq j \leq d-3$.
For example, when $d=4$ the first few terms are
\begin{align*}
  \beta_0 &= (0, 0) \\
  \beta_1 &= (0, 0), (1, -1) \\
  \beta_2 &= (0, 0), (1, -1), (2,-2) \\
          &\dots
\end{align*}
We can form a sequence, which we will also denote by $\beta$ but with only one index, by concatenating the terms above. Each $\beta_{n, (\cdot)}$ contains
$\sim n^{d-3}$ elements, so $\beta_0, \dots, \beta_{O(n^{d-2})}$ will have traversed all the elements of $\beta_{n, (\cdot)}$ (with repetitions).

Now we apply Lemma \ref{Lemma} to this sequence. In particular define for $1 \leq i \leq d-2$
\begin{equation*}
  \varphi_i(m) = \pi_i(\beta_m)
\end{equation*}
for $m \geq 0$ and $\varphi_i(m) = 0$ for all $m < 0$. Let $\Phi_i$ be the corresponding extensions provided by Lemma \ref{Lemma}.

For real $s$ let $\tau(s) = s^{d-2} \sign(s)$ where $\sign(s) = 1$ for positive $s$ and $-1$ for negative $s$.
Finally, consider the collection of $d$-tuples
\begin{equation*}
  \mathcal{A} = \bigl\{ \bigl( s, t, \Phi_1(\tau(s), \tau(t)), \dots, \Phi_{d-2}(\tau(s), \tau(t)) \bigr): s,t \in \tau^{-1}(\mathbb{Z}) \bigr\},
\end{equation*}
and let $\mathcal{E}$ be the collection of all dyadic intervals in $\mathbb{R}^d$ associated with $\mathcal{A}$, i.e.: dyadic intervals $R$ with
dimensions
\begin{equation*}
  2^s \times 2^t \times 2^{\Phi_1(\tau(s),\tau(t))} \times \dots \times 2^{\Phi_{d-2}(\tau(s),\tau(t))} \quad s,t \in \tau^{-1}(\mathbb{Z}).
\end{equation*}
Note that after a trivial change of variables, this basis consists of intervals with dimensions
\begin{equation*}
  s \times t \times \Psi_1(s,t) \times \dots \times \Psi_{d-2}(s,t) \quad s,t \in \mathcal{A}' \subset \mathbb{R}_+
\end{equation*}
for non-decreasing functions $\Psi_i$ and a subset $\mathcal{A}'$ of $\mathbb{R}_+$.

We will show
that he family $\mathcal{E}'$ consisting of those intervals $R$ in $\mathcal{E}$ with
\begin{equation*}
  |\pi_1(R)||\pi_2(R)| = 1
\end{equation*}
cannot differentiate $L \log L^{d-3}$. To prove this we show, as in \cite{SoriaCounterexample},
that a certain weak-type estimate cannot hold.

\begin{theorem}
  Let $\mathcal{M}$ be the maximal operator associated to the family of intervals $\mathcal{E}'$, i.e.:
  \begin{equation*}
    \mathcal{M}f(x) = \sup_{R \in \mathcal{E}'} \frac{\mathbbm{1}_R(x)}{|R|} \int_R |f(y)| \, dy.
  \end{equation*}

  Suppose that we have for all non-negative functions $f$
  \begin{equation} \label{WeakType}
    |\{x \in \mathbb{R}^d:\, \mathcal{M}f(x) \geq 1\}| \lesssim \int_{\mathbb{R}^d} f(x) \log(e+f(x))^{\alpha},
  \end{equation}
  then $\alpha \geq d-2$.
\end{theorem}
\begin{proof}
  For $k \in \mathbb{N}$ define
  \begin{equation*}
    f_k(x) = 2^{dk} \mathbbm{1}_{[0, 2^{-k})}(x).
  \end{equation*}

  Let $\overline{R}$ be any dyadic interval in $\mathbb{R}^{d-2}$ of the form
  \begin{equation*}
    [0, 2^{m_1}) \times \dots \times [0, 2^{m_{d-2}}),
  \end{equation*}
  with $0 \leq m_i \leq j/C_d$ for $i \leq d-3$, $1 \leq j \leq k$, and with $|R| =1$, where $C_d > 1$ is a large dimensional constant that we will choose later.
  Note that there exist $\sim j^{d-3}$ such intervals.

  Given $\overline{R}$ there exists an $n$ satisfying $C_d^{-1}(j-1)^{d-2} \leq n \leq j^{d-2}$ (for sufficiently large $C_d$) such that $\beta_n = (m_1, \dots, m_{d-2})$, so
  let $R$ be the interval in $\mathcal{E}'$
  \begin{equation*}
    \Bigl[0, 2^{\tau^{-1}(n)}\Bigr) \times \Bigl[0, 2^{-\tau^{-1}(n)}\Bigr) \times \overline{R}.
  \end{equation*}

  Since $\tau^{-1}(n) \leq j$, the interval $R$ must contain $[0,2^{-k})^d$, so
  \begin{equation*}
    \frac{1}{|R|} \int_R f_k \geq 1.
  \end{equation*}

  Now it is easy to show that the union these intervals has measure $\sim k^{d-2}$ (for example, one can notice that they are
  \emph{sparse}\footnote{
    By sparse here we mean that for each $R$ there exists a subset $E(R)$ with $|E(R)| \gtrsim |R|$ for all $R$ and such that the sets $E(R)$ are
    pairwise disjoint, so in particular the sum of their measures is comparable to the measure of the union.
  }), so
  \begin{equation*}
    |\{x \in \mathbb{R}^d:\, \mathcal{M}f(x) \geq 1\}| \gtrsim k^{d-2}
  \end{equation*}
  and the claim follows by testing $f_k$ in the right hand side of \eqref{WeakType}.
\end{proof}

Failure of this type of weak-type estimates implies that the the basis cannot differentiate $L \log L^{d-3}$, as shown
by I. Peral (see page 90 of \cite{GuzmanBook}).

\bibliography{article}
\bibliographystyle{abbrv}

\end{document}